\documentclass[12pt]{article}

\usepackage{epstopdf}
\usepackage{latexsym, amsmath, amscd,amsthm}
\usepackage{amssymb}
\usepackage{latexsym, amsmath, amscd,amsthm,booktabs,array,url}
\usepackage{graphicx,epsf}
\usepackage{xcolor,colortbl}

\newtheorem{theorem}{Theorem}[section]
\newtheorem{corollary}[theorem]{Corollary}
\newtheorem{proposition}[theorem]{Proposition}
\newtheorem{lemma}[theorem]{Lemma}

\theoremstyle{definition}

\newtheorem{remark}[theorem]{Remark}

\begin{document}

\title{Bounds for the geometric-arithmetic index of unicyclic graphs}
\author{Sunyo Moon\footnotemark[1]\,  
and Seungkook Park\footnotemark[2]} 

\date{}
\renewcommand{\thefootnote}{\fnsymbol{footnote}}
\footnotetext[1]{Research Institute of Natural Sciences, Hanyang University,Seoul, Republic of Korea(symoon89@hanyang.ac.kr)}
\footnotetext[2]{Department of Mathematics and Research Institute of Natural Sciences, Sookmyung Women's University, Seoul, Republic of Korea(skpark@sookmyung.ac.kr)}
\renewcommand{\thefootnote}{\arabic{footnote}}

\maketitle

\begin{abstract}
We present lower and upper bounds for the geometric-arithmetic index of unicyclic graphs and provide extremal graphs for the corresponding bounds.
\end{abstract}

\section{Introduction}

Among the numerous topological descriptors, Randi\'{c} index is the most used topological index for applications in chemistry and pharmacology. In \cite{Vukicevic_D}, Vuki\v{c}evi\'{c} and Furtula introduced a new topological index based on the ratios of geometrical and arithmetical means of endvertex degrees of edges, which is called the {\it geometric-arithmetic index} (GA index) and showed that the prediction power of GA index for physico-chemical properties such as entropy, enthalpy of vaporization, standard enthalpy of vaporization, enthalpy of formation, and acentric factor, is at least for 2.5\% better than that of Randi\'{c} index.  For these reasons GA index has been used in the development of QSPR/QSAR researches.

In recent years, many studies have been done on the GA index \cite{Aouchiche-geo, Aouchiche-adj, Aouchiche-com, Chen, Du,  Gutman, Rodriguez, Vujosevic}. A main topic in the study of GA index is to find bounds of the GA index involving several parameters. An upper bound for the GA index of trees in terms of the order and the domination number was given in \cite{Bermudo}.  Lower and upper bounds for the GA index in terms of minimal degree was presented in \cite{Aouchiche-geo}. An upper bound on the ratio of the GA index and the chromatic number of a graph was provided in \cite{AOUCHICHE2017207}. Lower bounds for the GA index in terms of number of vertices, number of edges, maximal vertex degree, minimal non-pendant vertex degree, number of pendant vertices, and the second Zagreb index was established in \cite{Das}. Vuki\v{c}evi\'{c} and Furtula gave lower and upper bounds of the GA index for trees and chemical trees in \cite{Vukicevic_D}. In this paper, we give lower and upper bounds of the GA index for unicyclic graphs and provide extremal graphs for the corresponding bounds.

\section{Upper bound of the GA index}

In this section, we introduce some definitions and inequalities. We also show that the upper bound for the GA index of the unicyclic graph $G$ with $n$ vertices is $n$ which is obtained by the cycle graph $C_n$.

Let $G=(V(G),E(G))$ be a simple graph with vertex set $V(G)$ and edge set $E(G)$.
The {\it degree} of a vertex $v$ of $G$ is the number of edges incident to $v$, denoted by $d_G(v)$.
The geometric-arithmetic index of $G$ is
\begin{equation*}
  GA(G) =\sum_{uv \in E(G)}\frac{2\sqrt{d_G(u)d_G(v)}}{d_G(u)+d_G(v)}
\end{equation*}
and the {\it arithmetic-geometric index} (AG index) is
\begin{equation*}
    AG(G)=\sum_{uv \in E(G)}\frac{d_G(u)+d_G(v)}{2\sqrt{d_G(u)d_G(v)}}
\end{equation*}
A connected graph $G$ is called a {\it unicyclic graph} if it contains exactly one cycle. In \cite{Vukicevic_Z}, Vuki\'{c}evi\'{c}, Vujo\v{s}evi\'{c}, and Popivoda computed the
lower and upper bounds for the AG index of unicyclic graphs and proved that the bounds are sharp by providing the extremal graphs for the corresponding bounds. Our goal is to find the lower and upper bounds for the GA index of a unicyclic graph $G$ and the extremal graphs for the corresponding bounds, that is, we show that
\begin{equation}\label{eq:main}
 1+\frac{(2n^2+4(\sqrt{2}-1)n-6)\sqrt{n-1}}{n(n+1)} \leq GA(G) \leq n,
\end{equation}
where $n$ is the number of vertices of the graph $G$. The extremal graph for the upper bound is the cycle graph $C_n$  and the extremal graph for the lower bound is the graph obtained by attaching $n-3$ pendant vertices to a vertex of the cycle graph $C_3$. 

We follow the definitions and notations in \cite{Vukicevic_Z}.
A vertex $v$ is called a {\it cycle vertex} if it lies on the cycle of $G$ and an edge $e$ is called a {\it cycle edge} if it is an edge of the cycle of $G$.
A cycle vertex $v$ is called a \textit{local maximum vertex} if the degree of $v$ is greater than or equal to the degree of its neighbor cycle vertices, that is, $d_G(v) \geq \max\{d_G(v_i):i=1,2\}$, where $v_1$, $v_2$ are the neighbor cycle vertices of $v$.
A cycle vertex $u$ is called a \textit{local minimum vertex} if the degree of $u$ is less than or equal to the degree of its neighbor cycle vertices, that is, $d_G(u) \leq \min\{d_G(u_i):i=1,2\}$, where $u_1$, $u_2$ are the neighbor cycle vertices of $u$.
For a cycle vertex $v$, the {\it pendant tree rooted} at $v$, denoted by $T_v$, is the maximal connected subgraph of $G$ which contains the vertex $v$ and no other cycle vertices.
We denote the set of all edges of $T_v$ in $G$ by $E_G(T_v)$ .
Let $e$ be an edge of $G$ and let $u$ and $v$ be endvertices of $e$.
Then the contribution of $e$ to $GA(G)$, denoted by $GA_G(e)$, is
\begin{equation*}
    \qquad GA_G(e)=\frac{2\sqrt{d_G(u)d_G(v)}}{d_G(u)+d_G(v)}.
\end{equation*}
Then the GA index can be written as
$$GA(G) =\sum_{e \in E(G)}GA_G(e).$$
Let $f(x)=\frac{2\sqrt{x}}{1+x}$.
Then
$$GA_G(e)=f\bigg(\frac{d_G(v)}{d_G(u)}\bigg).$$
Since the function $f$ is decreasing on $[1,\infty)$, we have $f(x) \leq 1$ and the equality holds if and only if $d_G(u)=d_G(v)$.
Throughout this paper, we assume that $d_G(u) \leq d_G(v)$ and denote the ratio of the endvertex degrees $\frac{d_G(v)}{d_G(u)}$ of $e$ as $rd_G(e)$. Since
\[
 \frac{2\sqrt{n-1}}{n} \leq GA_G(e) \leq 1,
\]
it is easy to see that the upper bound $n$ for the GA index is obtained by the cycle graph $C_n$.

\section{Lower bound of the GA index}
Let $G$ be a unicyclic graph with $n$ vertices.
If $n=3$, then $C_3$ is the unique unicyclic graph and the GA index of $C_3$ is the upper bound in inequality \eqref{eq:main}.
If $n=4$, then there are two unicyclic graphs $C_4$ and the paw graph. The GA index of $C_4$ is the upper bound and the GA index of the paw graph is the lower bound in inequality \eqref{eq:main}.
Thus from now on we make the assumption that all graphs are unicyclic with $n \geq 5$ vertices.

In order to find the lower bound for the GA index, we start with a general unicyclic graph then we cut and paste edges such that the resulting graph has smaller GA index. The following lemma shows that if the ratio of the endvertex degrees of an edge increases then the contribution of the edge to the GA index decreases.

\begin{lemma}\label{lem:ga-e}
  Given two graphs $G$ and $G'$, we let $e=uv$ and $e'=u'v'$ be edges of $G$ and $G'$, respectively.
  If $d_G(u) \geq d_{G'}(u')$ and $d_G(v) \leq d_{G'}(v')$, then
  \begin{equation*}
      GA_G(e) \geq GA_{G'}(e').
  \end{equation*}
\end{lemma}
\begin{proof}
    Since $f$ is decreasing on $[1, \infty)$ and
    $$1 \leq \frac{d_G(v)}{d_G(u)} \leq \frac{d_{G'}(v')}{d_{G'}(u')},$$
    we have $GA_G(e) \geq GA_{G'}(e')$.
\end{proof}
\begin{remark}
    Lemma \ref{lem:ga-e} can be restated as follows:
If $rd_G(e) \leq rd_{G'}(e')$ then $GA_G(e) \geq GA_{G'}(e')$.
\end{remark}

\begin{figure}[h!t!b!]
    \centering
    \includegraphics[width=6.5cm]{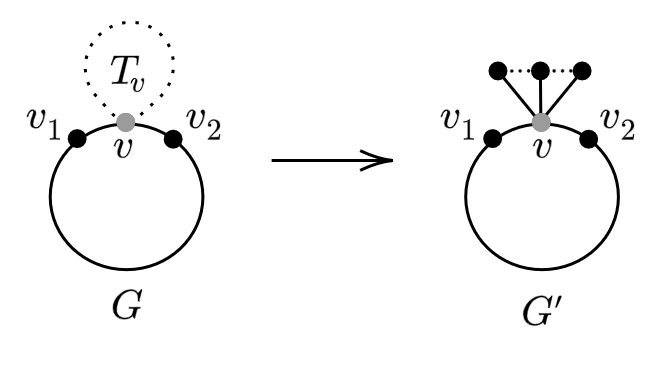}
    \caption{Transform $T_v$ into a star with center $v$}
    \label{fig:maxdeg-tree}
\end{figure}

Roughly speaking, the following proposition says that if we change the pendant tree of a cycle vertex to a star centered at that vertex then the GA index decreases.

\begin{proposition}\label{prop:maxdeg-tree}
 Let $G$ be a unicyclic graph with a local maximum vertex $v$ of $G$.
 If $G'$ is the unicyclic graph obtained from $G$ by transforming $T_v$ into a star graph with center vertex $v$(see Figure \ref{fig:maxdeg-tree}) then $GA(G') \leq GA(G).$
\end{proposition}
\begin{proof}
  We rearrange each edge $e$ in $E_G(T_v)$ to a pendant edge at $v$ in $G'$.
  Since
  \begin{equation*}
      rd_G(e) \leq \frac{\lvert E_G(T_v) \rvert +2}{1} = \frac{\lvert E_{G'}(T_v) \rvert +2}{1}=rd_{G'}(e),
  \end{equation*}
  by Lemma \ref{lem:ga-e}, $GA_G(e) \leq GA_{G'}(e)$.
  Let $v_1$ and $v_2$ be the neighbor cycle vertices of $v$.
  Now, we consider the contribution of the cycle edges $vv_1$ and $vv_2$.
  Since the degrees of $v_1$ and $v_2$ are unchanged and $d_G(v)\leq d_{G'}(v)$, we have
  \[
  1 \leq \frac{d_G(v)}{d_G(v_i)} \leq \frac{d_{G'}(v)}{d_{G'}(v_i)} ~~\text{for $i=1,2$.}
  \]
  By Lemma \ref{lem:ga-e},  $GA_G(vv_i) \geq GA_{G'}(vv_i)$ for $i=1,2$. Hence $GA(G) \geq GA(G')$.
\end{proof}

\begin{corollary}\label{coro:maxdeg-tree}
  Let $G$ be a unicyclic graph and let $v$ be a maximal degree cycle vertex of $G$.
  If $G'$ is the unicyclic graph obtained from $G$ by transforming $T_v$ into a star graph with center vertex $v$, then $GA(G') \leq GA(G)$.
\end{corollary}

\begin{figure}[h!t!b!]
    \centering
    \includegraphics[width=8cm]{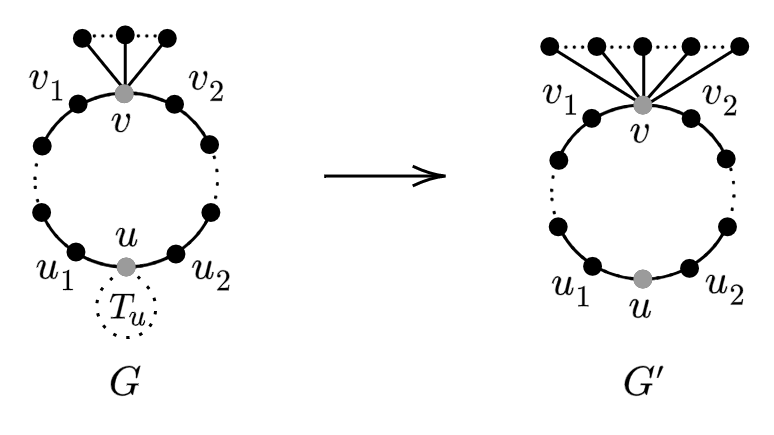}
    \caption{Edges in $E_G(T_u)$ relocated to pendant edges at $v$ in $G'$}
    \label{fig:mindeg}
\end{figure}

\begin{proposition}\label{prop:mindeg}
   Let $G$ be a unicyclic graph.
   Suppose that $v$ is a local maximum vertex of $G$ and $T_{v}$ is a star with center $v$.
   If $u$ is a local minimum vertex of $G$
   and $G'$ is the graph obtained from $G$ by relocating all the edges of the pendant tree $T_u$ to pendant edges at $v$ then $GA(G') \leq GA(G)$ (see Figure \ref{fig:mindeg}).
\end{proposition}
\begin{proof}
   We compare the GA index of $G$ and $G'$. In the graph $G'$, the degree of the cycle vertex $v$ is increased and the degree of the cycle vertex $u$ is decreased. However, the degrees of the other cycle vertices are not changed.
   Thus we only need to consider the edges of $T_u$, $T_v$ and the cycle edges incident to $u$, $v$.
   We compute the changes of the contribution of the edges of $T_u$ in $G$.
   Suppose that an edge $e$ in $E_G(T_u)$ is relocated to a pendant edge at $v$ in $G'$.
   For each $e \in E_G(T_u)$,
   \begin{equation*}
      rd_G(e) \leq \frac{\lvert E_G(T_u) \rvert +2}{1}\leq \frac{\lvert E_{G'}(T_v) \rvert +2}{1}=rd_{G'}(e).
   \end{equation*}
   Thus $GA_G(e) \geq GA_{G'}(e)$, by Lemma \ref{lem:ga-e}.
   For each $e \in E_G(T_v)$,
   \begin{equation*}
     rd_G(e)=\frac{d_G(v)}{1} \leq \frac{d_{G'}(v)}{1}=rd_{G'}(e).
   \end{equation*}
   Thus $GA_G(e) \geq GA_{G'}(e)$.
   Let $u_1$ and $u_2$ be the neighbor cycle vertices of $u$.
   Since $d_G(u) \geq d_{G'}(u)=2$ and $d_G(u_i) = d_{G'}
   (u_i)$ for $i=1,2$, the contribution of the cycle edges incident to $u$ decreases.
   Now, we consider the cycle edges incident to $v$.
   Let $v_1$ and $v_2$ be the neighbor cycle vertices of $v$.
   If $u \neq v_i$ for all $i=1,2$, then $d_G(v) \leq d_{G'}(v)$ and $d_G(v_i)=d_{G'}(v_i)$ for all $i=1,2$.
   If $u=v_i$ for some $i=1,2$, then $d_G(v) \leq d_{G'}(v)$ and $d_G(u) \geq d_{G'}(u)=2$.
   For both cases, the contribution of the cycle edges incident to $v$ decreases, by Lemma \ref{lem:ga-e}.
   Therefore, we have $GA(G') \leq GA(G)$.
\end{proof}

\begin{figure}[h!t!b!]
    \centering
    \includegraphics[width=8cm]{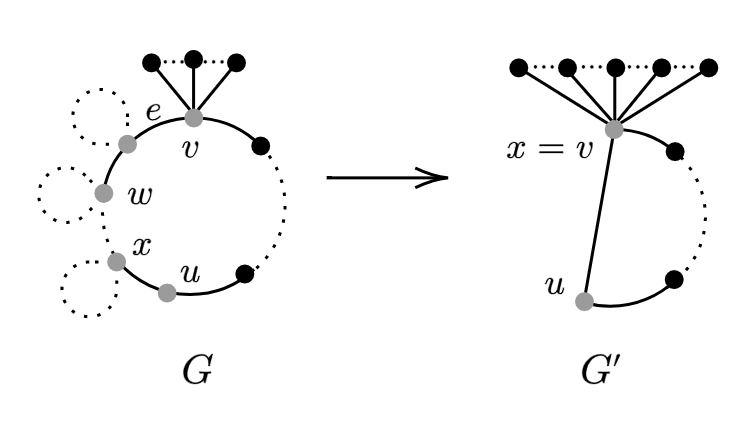}
    \caption{The $\operatorname{arc}(uev)$-transformation}
    \label{fig:arctrans}
\end{figure}

Let $u,v$ be two non-adjacent cycle vertices and let $e$ be a cycle edge.
We denote the $(u,v)$-arc containing the edge $e$ by $\operatorname{arc}(uev)$.
We define the \textit{$\operatorname{arc}(uev)$-transformation} from $G$ to $G'$ as follows:

\begin{itemize}
    \item[(1)] Suppose that $w$ is a cycle vertex, different from $u$ and $v$, lying on the $\operatorname{arc}(uev)$.
    Then we relocate all the edges in $E_G(T_w)$ to pendant edges at $v$.
    \item[(2)] Let $x$ be a cycle vertex on the $\operatorname{arc}(uev)$ adjacent to $u$.
    We relocate all the cycle edges on the  $\operatorname{arc}(uev)$ to pendant edges at $v$, except for the cycle edge $ux$, and then we identify $v$ with $x$ (see Figure \ref{fig:arctrans}).
\end{itemize}

\begin{proposition}\label{prop:arctrans}
    Let $G$ be a unicyclic graph.
    Let $u$, $v$ be non-adjacent cycle vertices and let $e$ be a cycle edge.
    Suppose that
    \begin{itemize}
        \item[(1)] $v$ is a local maximum vertex and $T_v$ is a star with center $v$ and
        \item[(2)] each cycle vertex $w$ lying on the $\operatorname{arc}(uev)$ satisfies $d_G(u) \leq d_G(w) \leq d_G(v)$.
    \end{itemize}
    If $G'$ is a graph obtained by the $\operatorname{arc}(uev)$-transformation, then $GA(G') \leq GA(G)$.
\end{proposition}
\begin{proof}
  Let $w$ be a cycle vertex lying on the $\operatorname{arc}(uev)$.
  For each edge $e$ in $E_G(T_w)$,
  $$rd_G(e) \leq \frac{\lvert E_G(T_w) \rvert +2}{1} \leq \frac{\lvert E_{G'}(T_v) \rvert+2}{1}=rd_{G'}(e).$$ By Lemma \ref{lem:ga-e}, we have $GA_G(e)\geq GA_{G'}(e)$.
  For a cycle edge on the $\operatorname{arc}(uev)$ not incident to $u$,
  let $a$ and $b$ be endvertices of the edge.
  We may assume that $d_G(a) \leq d_G(b)$.
  Then $d_G(a) \leq d_G(b) \leq d_G(v) \leq d_{G'}(v)$.
  Since $$\frac{d_G(b)}{d_G(a)} \leq \frac{d_{G'}(v)}{1},$$ the contribution of the edge decreases.
  Now, we consider the cycle edge $xu$.
  Since $$d_{G'}(u) = d_G(u) \leq d_G(x) \leq d_G(v) \leq d_{G'}(v),$$
  we have $$\frac{d_G(x)}{d_G(u)} \leq \frac{d_{G'}(v)}{d_{G'}(u)}.$$
  Thus the contribution of the cycle edge $xu$ decreases.
  For each edge in $E_G(T_v)$, the contribution of the edge decreases because
  the degree of $v$ is increased.
  Finally, we examine the contribution of the cycle edge incident to $v$ not lying on the $\operatorname{arc}(uev)$.
  Since the degree of $v$ is increased and the degree of the other endvertex is unchanged, the contribution of the cycle edge decreases.
\end{proof}

\begin{figure}[h!b!t!]
    \centering
    \includegraphics[width=9cm]{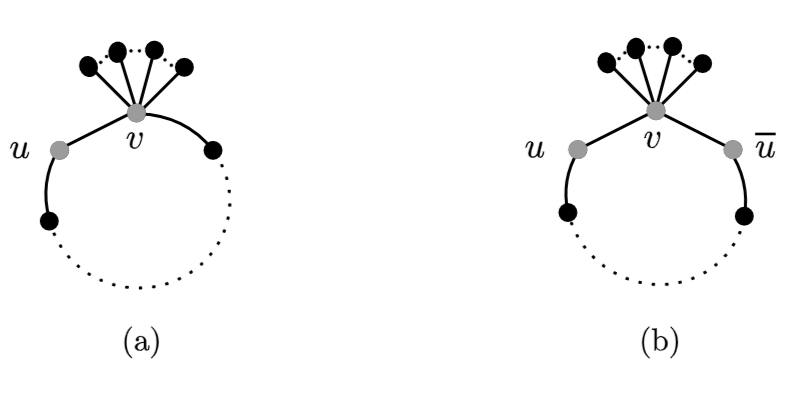}
    \caption{(a) $v$ has exactly one cycle vertex $u$ of degree 2 (b) $v$ has two neighbor cycle vertices $u$ and $\bar{u}$ of degree 2}
    \label{fig:case-nbd}
\end{figure}

Let $G$ be a unicyclic graph with $n$ vertices.
We perform the following steps to transform the unicyclic graph $G$ into a unicyclic graph $G'$ with a smaller GA index.

\begin{itemize}
    \item [1.] Let $v$ be a maximal degree cycle vertex. Let $v_1$ and $v_2$ be the neighbor cycle vertices of $v$.
    We rearrange all the edges in $E_G(T_v)$ to pendant edges at $v$. Then $T_v$ becomes a star with center $v$ in the new graph $G_1$.
    By Corollary \ref{coro:maxdeg-tree}, $G_1$ has a smaller GA index (see Figure \ref{fig:maxdeg-tree}).
    \item [2.] Let $u$ be a minimal degree cycle vertex in $G_1$. We relocate all the edges of $T_{u}$ to pendant edges at $v$.
    By Proposition \ref{prop:mindeg}, the new graph $G_2$ has a smaller GA index.
    We note that the degree of $u$ in $G_2$ is 2 (see Figure \ref{fig:mindeg}).
    \item[3.] We consider the following two cases.
    If $u$ is adjacent to $v$, then we put $G_3=G_2$ and go to step 4.
    Suppose $u$ is not adjacent to $v$. We let $e=vv_1$. By the $\operatorname{arc}(uev)$-transformation, we obtain a new graph $G_3$.
    By Proposition \ref{prop:arctrans}, $GA(G_3) \leq GA(G_2)$. In both cases, $u$ is a neighbor of $v$ with degree 2 in $G_3$ (see Figure \ref{fig:arctrans}).
    \item[4.]
    If there is no local minimum vertex other than $u$, then $v$ has exactly one neighbor cycle vertex of degree 2 (see Figure \ref{fig:case-nbd} (a)).
    Suppose that there exists a local minimum vertex $\bar{u}$ which is not $u$ in $G_3$.
    If $\bar{u}$ is adjacent to $v$, that is, $\bar{u}=v_2$,
    then by step 2, we transform the graph so that the neighbor cycle vertex $\bar{u}$ of $v$ has degree 2.
    The transformed graph has a smaller GA index  by Proposition \ref{prop:mindeg}.
    If $\bar{u}$ is not adjacent to $v$, then let $\bar{e}=vv_2$.
    By step 2 and the $\operatorname{arc}(\bar{u}\bar{e}v)$-transformation, we obtain a new graph with a smaller GA index.
    In both cases, $v$ has two neighbor cycle vertices $u$ and $\bar{u}$ of degree 2 (see Figure \ref{fig:case-nbd} (b)).

\end{itemize}

\begin{figure}[t!]
    \centering
    \includegraphics[width=10cm]{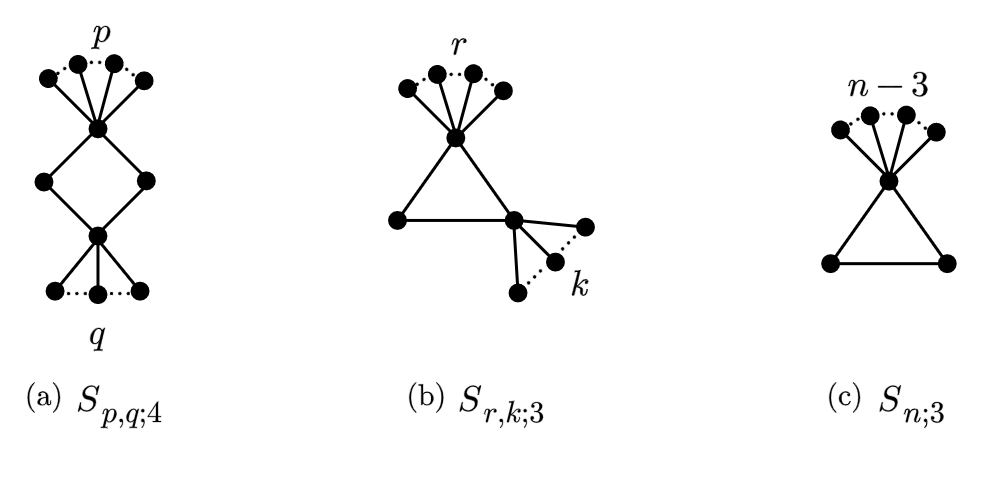}
    \caption{The graphs $S_{p,q;4}$, $S_{r,k;3}$ and $S_{n;3}$}
    \label{fig:graphs}
\end{figure}

Now, we consider the two graphs in Figure \ref{fig:case-nbd}. First, we introduce some notations.
The graph with $n$ vertices obtained by attaching $p$ and $q$ pendant vertices, respectively, to the two nonadjacent vertices of the cycle $C_4$ is denoted by
$S_{p,q;4}$ (see Figure \ref{fig:graphs} (a)).
The graph with $n$ vertices obtained by attaching $r$ and $k$ pendant vertices, respectively, to two vertices of the cycle $C_3$ is denoted by
$S_{r,k;3}$ (see Figure \ref{fig:graphs} (b)).
The graph with $n$ vertices obtained by attaching $n-3$ pendant vertices to a vertex of the cycle $C_3$ is denoted by $S_{n;3}$ (see Figure \ref{fig:graphs} (c)).

The following proposition shows that the GA index of $S_{p,q;4}$ is smaller than the graph in Figure \ref{fig:case-nbd} (b) of girth at least 4.
\begin{proposition}\label{prop:case-deg2}
    Let $G$ be an $n$-vertex unicyclic graph of girth at least 4.
    Let $v$ be a maximal degree cycle vertex and let $T_v$ be a star with center $v$.
    Suppose that $v$ has two neighbor cycle vertices of degree 2.
    Then $$GA(G) \geq GA(S_{p,q;4})$$ for some non-negative integers $p$ and $q$.
\end{proposition}
\begin{proof}
    Let $u$ and $\bar{u}$ be the neighbor cycle vertices of $v$.
    We denote the remaining cycle vertices by $v_1,\ldots,v_t$ for $t\geq 1$.
    If $t=1$, then we put $G'=G$.
    For $t\geq 2$, assume that $v_1$ is adjacent to $u$ and $v_t$ is adjacent to $\bar{u}$.
    Let $\bar{v}$ be the maximal cycle vertex in $\{v_1,\ldots,v_t\}$.
    If $\bar{v}$ is adjacent to $u$, then we let $\bar{e}=\bar{u}v_t$.
    By the $\operatorname{arc}(\bar{u}\bar{e}\bar{v})$-transformation, we obtain a graph $G'$ of girth 4.
    If $\bar{v}$ is adjacent to $\bar{u}$, then we let $e=uv_1$. By the $\operatorname{arc}(ue\bar{v})$-transformation, we can show that the transformed graph $G'$ has girth 4.
    Suppose that $\bar{v}$ is not adjacent to $u$ nor $\bar{u}$.
    Let $e_1=uv_1$ and $e_2=\bar{u}v_t$.
    Then by the $\operatorname{arc}(ue_1\bar{v})$-transformation and $\operatorname{arc}(\bar{u}e_2\bar{v})$-transformation, we have a new graph $G'$ of girth 4.
    By Proposition \ref{prop:arctrans}, $G'$ has a smaller GA index.

    Since $\bar{v}$ has the neighbor cycle vertices $u$ and $\bar{u}$, the cycle vertex $\bar{v}$ is a local maximum vertex.
    We rearrange all the edges of $T_{\bar{v}}$ to pendant edges at $\bar{v}$.
    Then the graph $G'$ is transformed to $S_{p,q;4}$ for some non-negative integers $p,q$.
    By Proposition \ref{prop:maxdeg-tree}, the GA index of $S_{p,q;4}$ is smaller than the GA index of $G'$.
    Thus we have
    \begin{equation*}
        GA(G) \geq GA(G') \geq GA(S_{p,q;4}).
    \end{equation*}
\end{proof}

\begin{remark}
  The condition that the girth is at least 4 in Proposition \ref{prop:case-deg2} is due to the fact that if the girth is 3, then $G$ is equal to  $S_{n;3}$.
\end{remark}

The following proposition states that the graph in Figure \ref{fig:case-nbd} (a) can be transformed into either $S_{p,q;4}$ or $S_{r,k;3}$ which has a smaller GA index.

\begin{proposition}
    Let $G$ be a unicyclic graph with $n$ vertices.
    Let $v$ be a maximal degree cycle vertex and let $T_v$ be a star with center $v$.
    Suppose that $v$ has exactly one neighbor cycle vertex $u$ with degree 2 and there is no local minimum vertex other than $u$.
    Then $$GA(G) \geq GA(S_{p,q;4}) ~~\text{or}~~GA(G) \geq GA(S_{r,k;3})$$ for some non-negative integers $p$, $q$, $r$ and $k$.
\end{proposition}

\begin{proof}
    Let $u$ be the neighbor cycle vertex of $v$ with degree 2.
    Assume that the remaining cycle vertices are $v_1,\ldots,v_t$ in the counterclockwise order, where $t\geq 1$.
    If $t=1$, then we put $G'=G$.
    Assume that $t \geq 2$.
    Since there is no local minimum vertex other than $u$, we have
    \begin{equation*}
        d_G(v_i) \leq d_G(v_{i+1}) ~~\text{for each $i=1,\ldots t-1$.}
    \end{equation*}
    Hence the maximal degree cycle vertex in $\{v_1,\ldots,v_t\}$ is $v_t$.
    Let $e=uv_1$. Then we can reduce the girth of the graph to 3 by the $\operatorname{arc}(uev_t)$-transformation.
    We denote this graph by $G'$. Now, we consider the following two cases:
    \begin{itemize}
        \item [1.] Every neighbor vertex $w$ of $v_t$ in $V_{G'}(T_{v_t})$ satisfies $d_{G'}(w) \leq d_{G'}(v_t)$.
        \item [2.] There exist a neighbor vertex $w$ of $v_t$ in $V_{G'}(T_{v_t})$ such that $d_{G'}(w) > d_{G'}(v_t)$.
    \end{itemize}

    \textbf{Case 1.}
    Suppose that every neighbor vertex $w$ of $v_t$ in $V_{G'}(T_{v_t})$ satisfies $d_{G'}(w) \leq d_{G'}(v_t)$.
    We relocate all the edges not adjacent to $v_t$ in $E_{G'}(T_{v_t})$ to pendant edges at $v$.
    Then $G'$ is transformed into $S_{r,k;3}$ for some non-negative integers $r$ and $k$ (see Figure \ref{fig:case-deg1-1}).
    \begin{figure}[b!t!h!]
        \centering
        \includegraphics[width=8.5cm]{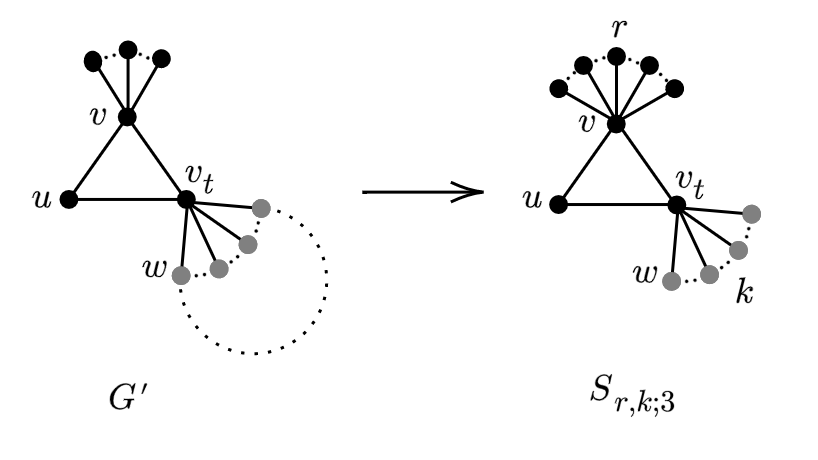}
        \caption{Transformation of the graph $G'$ to $S_{r,k;3}$}
        \label{fig:case-deg1-1}
    \end{figure}

    Now, we compare the contributions of the edges of $G'$ and $S_{r,k;3}$.
    Let $e$ be a pendant edge in $E_{G'}(T_v)$.
    Since $d_{G'}(v) \leq d_{S_{r,k;3}}(v)$,
    \begin{equation*}
        rd_{G'}(e)=\frac{d_{G'}(v)}{1} \leq \frac{d_{S_{r,k;3}}(v)}{1}=rd_{S_{r,k;3}}(e).
    \end{equation*}
    Let $e$ be an edge not incident to $v_t$ in $E_{G'}(T_{v_t})$.
    Since
    \begin{align*}
        \lvert E_{S_{r,k;3}}(T_v) \rvert &=
        \lvert E_{G'}(T_v) \rvert +
        \lvert E_{G'}(T_{v_t}) \rvert -(d_{G'}(v_t)-2)\\
        & \geq \lvert E_{G'}(T_{v_t}) \rvert -(d_{G'}(v_t)-2),
    \end{align*}
    we have
    \begin{equation*}
        rd_{G'}(e) \leq \frac{\lvert E_{G'}(T_{v_t}) \rvert -(d_{G'}(v_t)-2)+1}{1} \leq \frac{\lvert
        E_{S_{r,k;3}}(T_v) \rvert+2}{1}=rd_{S_{r,k;3}}(e).
    \end{equation*}
    Let $e$ be an edge incident to $v_t$ in $E_{G'}(T_{v_t})$.
    Then
    \begin{equation*}
        rd_{G'}(e) \leq \frac{d_{G'}(v_t)}{1}= \frac{d_{S_{r,k;3}}(v_t)}{1}=rd_{S_{r,k;3}}(e).
    \end{equation*}
    For the two cycle edges incident to $v$,
    since the degree of $v$ is increased and the degrees of the other cycle vertices are fixed, the contribution of the edges is decreased.
    Thus $$GA(G) \leq GA(G') \leq GA(S_{r,k;3})$$ for some non-negative integers $r$ and $k$.

    \textbf{Case 2.} Suppose that $w$ is a neighbor vertex of $v_t$ in $V_{G'}(T_{v_t})$ such that $d_{G'}(w) > d_{G'}(v_t)$.
    Let $T_w \subset T_{v_t}$ be a maximal tree rooted at $w$ that does not contain the edge $wv_t$.
    First, we shift the edge $v_tu$ to the edge $wu$.
    Then the girth of resulting graph is 4.
    We relocate all the edges in $E_{G'}(T_{v_t})\backslash (E_{G'}(T_w) \cup \{wv_t\})$ to pendant edges at $v$ and then we rearrange the edges in $E_{G'}(T_w)$ to pendant edges at $w$.
    Then we obtain the graph $S_{p,q;4}$ for some positive integers $p$ and $q$ (see Figure \ref{fig:case-deg1-2}).
    \begin{figure}[h!t!b!]
     \centering
     \includegraphics[width=13cm]{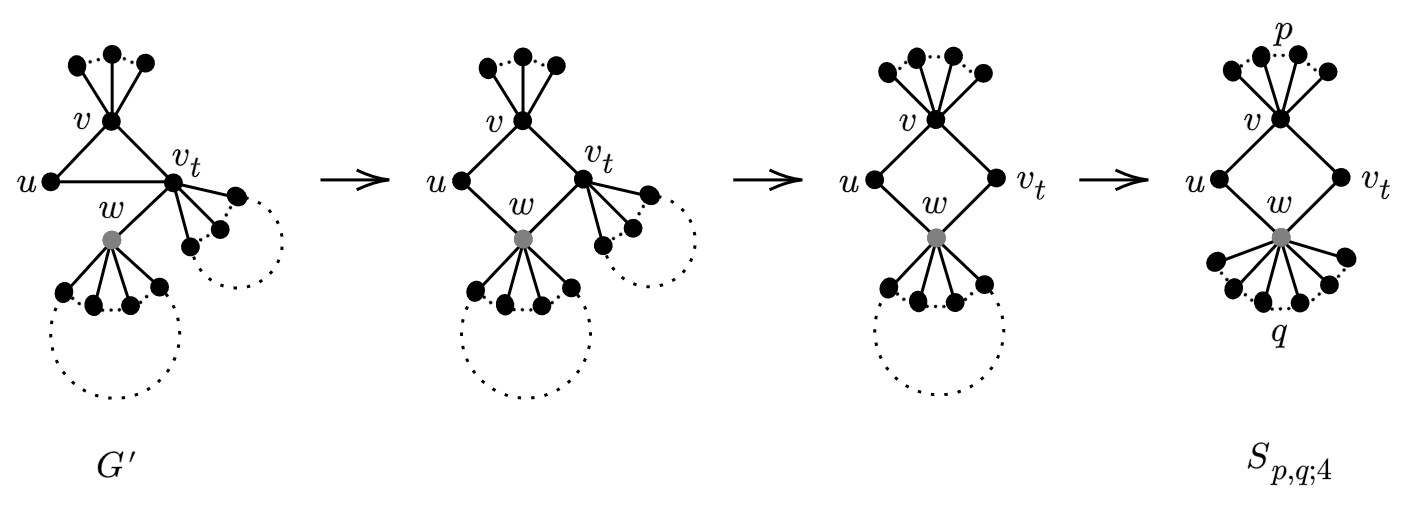}
     \caption{Transformation of the graph $G'$ to $S_{p,q;4}$}
     \label{fig:case-deg1-2}
    \end{figure}

    We compare the GA index of $G'$ and $S_{p,q;4}$.
    Since $d_{G'}(v) \leq d_{S_{p,q;4}}(v)$, the contribution of the edge in $E_{G'}(T_v)$ decreases.
    Let $e$ be an edge in $E_{G'}(T_{v_t})\backslash (E_{G'}(T_w) \cup \{wv_t\})$. Since
    \begin{align*}
         \lvert E_{S_{p,q;4}}(T_v) \rvert &=
        \lvert E_{G'}(T_v) \rvert +
        \lvert E_{G'}(T_{v_t}) \rvert -( \lvert E_{G'}(T_w)  \rvert +1)\\
        & > \lvert E_{G'}(T_{v_t}) \rvert -( \lvert E_{G'}(T_w)  \rvert +1),
    \end{align*}
    it follows that
    \begin{equation*}
        rd_{G'}(e) \leq \frac{\lvert E_{G'}(T_{v_t}) \rvert -( \lvert E_{G'}(T_w)  \rvert +1)+3}{1} \leq \frac{\lvert E_{S_{p,q;4}}(T_v) \rvert+2}{1}=rd_{S_{p,q;4}}(e).
    \end{equation*}
    Let $e$ be an edge in  $E_{G'}(T_w)$.
    Then
    \begin{equation*}
        rd_{G'}(e) \leq \frac{\rvert E_{G'}(T_w)\lvert  +1}{1} \leq \frac{\rvert E_{S_{p,q;4}}(T_w) \lvert+2}{1} =rd_{S_{p,q;4}}(e)
    \end{equation*}
    because $\rvert E_{G'}(T_w)\lvert=\rvert E_{S_{p,q;4}}(T_w) \lvert$.
    We compare the ratios of the endvertex degrees of the edges $v_tu$ and $wu$.
    Since $d_{G'}(v_t) < d_{G'}(w) \leq d_{S_{p,q;4}}(w)$, we have
    \begin{equation*}
        rd_{G'}(v_tu)=\frac{d_{G'}(v_t)}{2} < \frac{d_{S_{p,q;4}}(w)}{2} =rd_{S_{p,q;4}}(wu).
    \end{equation*}
    For the edge $wv_t$,
    since $ d_{G'}(w) \leq d_{S_{p,q;4}}(w)$ and $d_{G'}(v_t) \geq d_{S_{p,q;4}}(v_t)=2$, we obtain
    \begin{equation*}
        rd_{G'}(wv_t) =\frac{d_{G'}(w)}{d_{G'}(v_t)} \leq \frac{d_{S_{p,q;4}}(w)}{2}=rd_{S_{p,q;4}}(wv_t).
    \end{equation*}
    For the edge $vv_t$,
    since $d_{G'}(v) \leq d_{S_{p,q;4}}(v)$ and $d_{G'}(v_t) >d_{S_{p,q;4}}(v_t)=2$, we have
    \begin{equation*}
        rd_{G'}(vv_t)=\frac{d_{G'}(v)}{d_{G'}(v_t)} \leq \frac{d_{S_{p,q;4}}(v)}{2} =rd_{S_{p,q;4}}(vv_t).
    \end{equation*}
    For the edge $vu$, the degree of $v$ is increased and the degree of $u$ is fixed.
    Hence the contribution of $vu$ decreases.
    Thus $$GA(G) \leq GA(G') \leq GA(S_{p,q;4})$$ for some non-negative integers $p$ and $q$.
\end{proof}

In the following propositions, we will prove that the GA index of $S_{n;3}$ is less than the GA index of $S_{p,q;4}$ and the GA index of  $S_{r,k;3}$ for some non-negative integers $p$, $q$, $r$ and $k$.

\begin{proposition}\label{prop:GA_case1}
 Let $n$, $p$ and $q$ be non-negative integers with $n=p+q+4$. Then
 $GA(S_{n,3}) < GA(S_{p,q;4})$
 for $n \geq 5$.
\end{proposition}
\begin{proof}
  We may assume that $p\geq q \geq 0$.
  Note that $$GA(S_{p,q;4})=p\frac{2\sqrt{p+2}}{p+3}+2\frac{2\sqrt{2}\sqrt{p+2}}{p+4}+2\frac{2\sqrt{2}\sqrt{q+2}}{q+4}+q\frac{2\sqrt{q+2}}{q+3}$$
  and
  $$GA(S_{n,3})=(n-3)\frac{2\sqrt{n-1}}{n}+2\frac{2\sqrt{2}\sqrt{n-1}}{n+1}+1.$$
  Let $f(x)=\frac{2\sqrt{x}}{x+1}$ and $g(x)=\frac{2\sqrt{2}\sqrt{x}}{x+2}$.
  Then $f$ is decreasing on $[1,\infty)$, $g$ is decreasing on $[2,\infty)$ and $g(x)>f(x)$ on $[2,\infty)$.
  It follows that
  \begin{align*}
      GA(S_{p,q;4})-GA(S_{n,3})=&\, p(f(p+2)-f(p+q+3))+q(f(q+2)-f(p+q+3))\\
      &\,+2g(p+2)+2g(q+2)-2g(p+q+3)-f(p+q+3)\\
      &\,-1.
  \end{align*}
  We need to show that $GA(S_{p,q;4})-GA(S_{n,3})>0$.
  Let
  \begin{equation*}
      A(p,q)=p(f(p+2)-f(p+q+3))+q(f(q+2)-f(p+q+3))
  \end{equation*}
  and
  \begin{equation*}
      B(p,q)=2g(p+2)+2g(q+2)-2g(p+q+3)-f(p+q+3).
  \end{equation*}
  It is easy to check that $A(p,q)>0$ and $B(p,q)>0$.
  Thus it is enough to show that  $A(p,q) \geq 1$ or $B(p,q) \geq 1$.
    First, we observe that $GA(S_{n,3})< GA(S_{p,q;4})$ for $q=0,1$.
  Since
  \begin{align*}
      B(p,0)=&\,2g(p+2)+2g(2)-2g(p+3)-f(p+3)\\
      =&\,2(g(p+2)-g(p+3))+(1-f(p+3))+1 >1
  \end{align*}
  and
  \begin{align*}
      B(p,1)=&\,2g(p+2)+2g(3)-2g(p+4)-f(p+4)\\
      =&\, 2(g(p+2)-g(p+4))+2g(3)-f(p+4)\\
      >&\,2g(3)-f(5) \approx 1.2142,
  \end{align*}
  we have $GA(S_{n,3})< GA(S_{p,q;4})$ for $q=0,1$.
  For a fixed $q \geq 2$, we show that $A(p,q)$ is an increasing function with respect to $p$.
  The partial derivative of $A(p,q)$ with respect to $p$ is
  \begin{align}\label{eq:part-1}
     \frac{\partial}{\partial p} A(p,q)= &\,\frac{p^2+9p+12}{(p+3)^2\sqrt{p+2}}-\frac{(p+q)^2+12(p+q)+24}{(p+q+4)^2\sqrt{p+q+3}} \nonumber\\
      \geq&\,\frac{p^2+9p+12}{(p+3)^2\sqrt{p+2}}-\frac{(p+2)^2+12(p+2)+24}{(p+6)^2\sqrt{p+5}}.
  \end{align}
 By reducing to a common denominator, the numerator becomes
 \begin{equation*}
     p^8 + 77p^7 + 1837p^6 + 21007p^5 + 133366p^4 + 492168p^3 + 1036800p^2 + 1135728p + 495072.
 \end{equation*}
Hence $\frac{\partial}{\partial p} A(p,q)>0$ for a fixed $q \geq 2$.
\begin{table}[t]
  \begin{center}
    \begin{tabular}{ |c||c|c||c|c||c|c| }
    \hline
    $p$&$A(p,2)$ & $B(p,2)$&$A(p,3)$ & $B(p,3)$& $A(p,4)$& $B(p,4)$ \\
    \hline
    2 &\cellcolor{black!20}  0.5543& \cellcolor{black!20} 1.4468 &- &- &-&-\\
    3& \cellcolor{black!20} 0.6934&\cellcolor{black!20} 1.4641 & \cellcolor{black!20} 0.8721 &\cellcolor{black!20}1.4713&-&- \\
    4& \cellcolor{black!20} 0.7994&\cellcolor{black!20} 1.4749 & 1.0108&1.4834 &1.1767&1.4681\\
    5& \cellcolor{black!20} 0.8825&\cellcolor{black!20} 1.4829 & 1.1211& 1.4740&1.3102&1.4624\\
    6& \cellcolor{black!20} 0.9491 &\cellcolor{black!20} 1.4896 & 1.2109&1.4744 &1.4199&1.4572\\
    7& 1.0036 & 1.4957 &1.2853 & 1.4750&1.5116&1.4531\\
    \hline
    \end{tabular}
    \vspace{0.5cm}
  \caption{Values of $A(p,q)$ and $B(p,q)$}
  \label{tab:AB-1}
  \end{center}
\end{table}
Since $A(7,2)>1$ and $A(p,2)$ is an increasing function with respect to $p$, we obtain $A(p,2)>1$ for $p \geq 7$. If $p<7$, then $B(p,2)>1$ (see Table \ref{tab:AB-1}).
Hence $GA(S_{n,3}) < GA(S_{p,q;4})$ for $p \geq q=2$.
Similarly, $GA(S_{n,3}) < GA(S_{p,q;4})$ for $p \geq q=3$.
Since $A(p,q)$ is an increasing function with respect to $p$, in order to prove $A(p,q)>1$ for $p \geq q \geq 4$, it is enough to show that  $A(q,q)>1$ for $q \geq 4$.
From Table \ref{tab:AB-1}, we see that $A(4,4)>1$.
Now, we suppose that $p=q \geq 5$.
We note that $$A(q,q)=2q(f(q+2)-f(2q+3)).$$
By the mean value theorem, there exists $c \in (q+2, 2q+3)$ such that $$f(2q+3)-f(q+2)=f'(c)(q+1).$$
Hence $A(q,q)=-2q(q+1)f'(c)$.
Since $f'$ is a negative increasing function on $[2,\infty)$, we have
  \begin{align*}
      A(q,q)=&\,-2q(q+1)f'(c)\\
      >&\,-2q(q+1)f'(2q+3)=\frac{q(q+1)^2}{(q+2)^2\sqrt{2q+3}}\\
      \geq &\, \frac{5(5+1)^2}{(5+2)^2\sqrt{2\cdot 5+3}} \approx 1.0188.
  \end{align*}
  Thus $GA(S_{n,3})< GA(S_{p,q;4})$ for $p \geq q \geq 4$.
\end{proof}

\begin{proposition}
 Let $n,r$ and $k$ be positive integers with $n=r+k+3$. Then
 $GA(S_{n,3})< GA(S_{r,k;3})$ for $r,k \geq 1$.
\end{proposition}
\begin{proof}
  We may assume that $r \geq k\geq 1$.
  Note that
  \begin{equation*}
      GA(S_{r,k;3})=r\frac{2\sqrt{r+2}}{r+3}+\frac{2\sqrt{2}\sqrt{r+2}}{r+4}+\frac{2\sqrt{2}\sqrt{k+2}}{k+4} +k\frac{2\sqrt{k+2}}{k+3}+\frac{2\sqrt{r+2}\sqrt{k+2}}{r+k+4}
  \end{equation*}
  and
  \begin{equation*}
      GA(S_{n,3})=(n-3)\frac{2\sqrt{n-1}}{n}+2\frac{2\sqrt{2}\sqrt{n-1}}{n+1}+1.
  \end{equation*}
  Let $f(x)=\frac{2\sqrt{x}}{x+1}$ and $g(x)=\frac{2\sqrt{2}\sqrt{x}}{x+2}$ which is defined in Proposition \ref{prop:GA_case1}.
  Then
  \begin{align*}
      GA(S_{r,k;4})-GA(S_{n,3})=&\,r(f(r+2)-f(r+k+2))+k(f(k+2)-f(r+k+2))\\
      &\,+g(r+2)+g(k+2)-2g(r+k+2)+\frac{2\sqrt{r+2}\sqrt{k+2}}{r+k+4}\\
      &\,-1.
  \end{align*}
  Let
  $$C(r,k)=r(f(r+2)-f(r+k+2))+k(f(k+2)-f(r+k+2))$$ and $$D(r,k)=g(r+2)+g(k+2)-2g(r+k+2)+\frac{2\sqrt{r+2}\sqrt{k+2}}{r+k+4}.$$
  Then $C(r,k) >0$ and $D(r,k)>0$.
  We show that $C(r,k) \geq 1$ or $D(r,k) \geq 1$.
\begin{table}[t]
  \begin{center}
    \begin{tabular}{ |c||c|c||c|c||c|c||c|c| }
    \hline
    $r$&$C(r,2)$ & $D(r,2)$&$C(r,3)$ & $D(r,3)$& $C(r,4)$ & $D(r,4)$& $C(r,5)$& $D(r,5)$ \\
    \hline
    2 & \cellcolor{black!20}0.4006 & \cellcolor{black!20}1.1536  &- &- &- &- &-&-\\
    3&\cellcolor{black!20}0.5289 &\cellcolor{black!20}1.1772 &\cellcolor{black!20}0.7009 &\cellcolor{black!20}1.2070 &- &- &-&-\\
    4&\cellcolor{black!20}0.6282 &\cellcolor{black!20}1.1886 &\cellcolor{black!20}0.8355 &\cellcolor{black!20}1.2226&\cellcolor{black!20}0.9992 &\cellcolor{black!20}1.2413 &-&- \\
    5&\cellcolor{black!20}0.7072 &\cellcolor{black!20}1.1936 &\cellcolor{black!20}0.9436 &\cellcolor{black!20}1.2303 &1.1317 &1.2513 &1.2850&1.2633 \\
    6&\cellcolor{black!20}0.7716 &\cellcolor{black!20}1.1949 &1.0324 &1.2333 &1.2413 &1.2561 &1.4126&1.2695\\
    7&\cellcolor{black!20}0.8251 &\cellcolor{black!20}1.1941 &1.1067 &1.2335 &1.3336 &1.2575 &1.5205&1.2721 \\
    8&\cellcolor{black!20}0.8703 &\cellcolor{black!20}1.1920 &1.1699 &1.2319 &1.4124&1.2568 &1.6133&1.2724\\
    9&\cellcolor{black!20}0.9091 &\cellcolor{black!20}1.1891 &1.2244 &1.2293 &1.4808 &1.2546 &1.6939&1.2710\\
    10&\cellcolor{black!20}0.9427 &\cellcolor{black!20}1.1858 &1.2719 &1.2259 &1.5406 &1.2516 &1.7647&1.2685\\
    11&\cellcolor{black!20}0.9723 &\cellcolor{black!20}1.1823 &1.3137 &1.2221 &1.5934 &1.2480 & 1.8276&1.2653\\
    12&\cellcolor{black!20}0.9984 &\cellcolor{black!20}1.1786 &1.3509 &1.2181 &1.6406 &1.2440 &1.8837&1.2616\\
    13&1.0218 &1.1750 &1.3842 &1.2139 & 1.6829&1.2397 &1.9343&1.2575\\
    \hline
    \end{tabular}
    \vspace{0.5cm}
  \caption{Values of $C(r,k)$ and $D(r,k)$}
  \label{tab:b}
  \end{center}
\end{table}
Since the partial derivative $C(r,k)$ with respect to $r$ is
\begin{equation*}
    \frac{\partial}{\partial r}C(r,k)= \frac{r^2+9r+12}{(r+3)^2\sqrt{r+2}}- \frac{(r+k)^2+9(r+k)+12}{(r+k+3)^2\sqrt{r+k+2}}>0,
\end{equation*}
$C(r,k)$ is an increasing function with respect to $r$ for a fixed $k\geq 1$.
We observe that $GA(S_{n;3}) < GA(S_{p,q;4})$ for $r \geq k =1$. Since
\begin{align*}
D(r,1)=&\,g(r+2)+g(3)-2g(r+3)+\frac{2\sqrt{3}\sqrt{r+2}}{r+5}\\
    >&\, \frac{2\sqrt{6}}{5}-\frac{2\sqrt{2}\sqrt{r+3}}{r+5}+\frac{2\sqrt{3}\sqrt{r+2}}{r+5}\\
    >&\,\frac{2\sqrt{6}}{5} \approx 0.9798
\end{align*}
and $C(r,1) \geq C(1,1) \approx 0.1321$, we have $C(r,1)+D(r,1)>1$.
Hence  $GA(S_{n;3}) < GA(S_{p,q;4})$ for $r \geq k =1$.
Since $C(13,2)>1$ and $C(r,2)$ is an increasing function with respect to $r$, we obtain $C(r,2)>1$ for $r \geq 13$.
If $r<13$, then $D(r,2)>1$ (see Table \ref{tab:b}).
Hence $GA(S_{n;3}) < GA(S_{p,q;4})$ for $r \geq k =2$.
Similarly, $GA(S_{n;3}) < GA(S_{p,q;4})$ for $r \geq k =3$ or $4$.
For $k \geq 5$,
it is enough to show that  $C(k,k)>1$.
From Table \ref{tab:b}, we see that $C(5,5)\approx 1.2850$.
Now, we show that
\begin{equation*}
    C(k,k)=2k(f(k+2)-f(2k+2))>1 ~~\text{for $k \geq 6$}.
\end{equation*}
By mean value theorem, there exists $c\in(k+2,2k+2)$ such that $$f(2k+2)-f(k+2)=f'(c)k.$$
Since $f'$ is a negative increasing function on $[2,\infty)$, we have
\begin{align*}
    C(k,k)=&\,-2k^2f'(c)\\
    >&-2k^2f'(2k+2) =2k^2\frac{2k+1}{(2k+3)^2\sqrt{2k+2}}\\
    \geq &\,\,2\cdot 6^2\frac{13}{15^2\sqrt{14}} \approx 1.1118.
\end{align*}
Hence $GA(S_{n;3}) < GA(S_{p,q;4})$ for $r \geq k =6$.

\end{proof}

By the previous results we have the following theorem.
\begin{theorem}
Let $G$ be a unicyclic graph with $n$ vertices. Then
\begin{equation*}
     GA(S_{n,3})\leq GA(G) \leq GA(C_n),
\end{equation*}
where
\begin{equation*}
    GA(S_{n,3})=1+\frac{(2n^2+4(\sqrt{2}-1)n-6)\sqrt{n-1}}{n(n+1)}~~\text{and}~~GA(C_n)=n.
\end{equation*}
\end{theorem}

\bibliographystyle{plain} 

\end{document}